\documentclass[11pt]{amsart}

\author[C.~Sanna]{Carlo Sanna$^\dagger$}
\address{\parbox{\linewidth}{
Politecnico di Torino, Department of Mathematical Sciences\\
Corso Duca degli Abruzzi 24, 10129 Torino, Italy\\[-8pt]}}
\email{carlo.sanna.dev@gmail.com}
\thanks{$\dagger\,$C.~Sanna is a member of GNSAGA of INdAM and of CrypTO, the group of Cryptography and Number~Theory of Politecnico di Torino}

\title{On the least common multiple of shifted powers}

\usepackage{amsmath}
\usepackage{amssymb}
\usepackage{amsthm}
\usepackage{booktabs}
\usepackage{geometry}
\usepackage{cite}
\usepackage{color}
\usepackage{hyperref}
\usepackage{enumerate}
\hypersetup{colorlinks=true}
\allowdisplaybreaks

\newtheorem{theorem}{Theorem}[section]

\newtheorem{lemma}[theorem]{Lemma}
\theoremstyle{remark}
\newtheorem{remark}{Remark}[section]
\newtheorem{question}{Question}[section]

\DeclareMathOperator{\lcm}{lcm}
\DeclareMathOperator{\Li}{Li}

\begin{document}

\begin{abstract}
Let $a \geq 2$ be an integer.
We prove that for every periodic sequence $(s_n)_{n \geq 1}$ in $\{-1, +1\}$ there exists an effectively computable rational number $C_\mathbf{s} > 0$ such that
\begin{equation*}
\log\lcm(a + s_1, a^2 + s_2, \dots, a^n + s_n) \sim C_\mathbf{s} \cdot \frac{\log a}{\pi^2} \cdot n^2 ,
\end{equation*}
as $n \to +\infty$, where $\lcm$ denotes the least common multiple.

Furthermore, we show that if $(s_n)_{n \geq 1}$ is a sequence of independent and uniformly distributed random variables in $\{-1, +1\}$ then
\begin{equation*}
\log\lcm(a + s_1, a^2 + s_2, \dots, a^n + s_n) \sim 6 \Li_2\!\big(\tfrac1{2}\big) \cdot \frac{\log a}{\pi^2} \cdot n^2 ,
\end{equation*}
with probability $1 - o(1)$, as $n \to +\infty$, where $\Li_2$ is the dilogarithm function.
\end{abstract}

\maketitle

\section{Introduction}

Let $(F_n)_{n \geq 1}$ be the sequence of Fibonacci numbers, defined recursively by $F_1 = 1$, $F_2 = 1$, and $F_{n + 2} = F_{n + 1} + F_n$, for every integer $n \geq 1$.
Matiyasevich and Guy~\cite{MR1712797} proved that
\begin{equation*}
\log \lcm (F_1, F_2, \dots, F_n) \sim \frac{3 \log\!\left(\tfrac{1 + \sqrt{5}}{2}\right)}{\pi^2} \cdot n^2 ,
\end{equation*}
as $n \to +\infty$, where $\lcm$ denotes the least common multiple.
This result was generalized to Lucas sequences, Lehmer sequences, and other sequences with special divisibility properties~\cite{MR1077711, MR1242715, MR1394375, MR3150887, MR1078087, MR1229668, MR993902, MR1114366}.
In~particular, for every integer $a \geq 2$ we have
\begin{equation}\label{equ:plus-formula}
\log\lcm(a - 1, a^2 - 1, \dots, a^n - 1) \sim \frac{3 \log a}{\pi^2} \cdot n^2
\end{equation}
and
\begin{equation}\label{equ:minus-formula}
\log\lcm(a + 1, a^2 + 1, \dots, a^n + 1) \sim \frac{4 \log a}{\pi^2} \cdot n^2 ,
\end{equation}
as $n \to +\infty$.
Precisely, \eqref{equ:plus-formula} follows from~\cite[Lemma~3]{MR993902} applied to the Lucas sequence $\left(\frac{a^n - 1}{a - 1}\right)_{n \geq 1}$; while \eqref{equ:minus-formula} follows from~\cite[Th\'eor\`eme]{MR1078087} applied to the companion Lucas sequence $(a^n + 1)_{n \geq 1}$.

We generalize \eqref{equ:plus-formula} and \eqref{equ:minus-formula} by giving asymptotic formulas for least common multiples of sequences of \emph{shifted powers} $(a^n + s_n)_{n \geq 1}$, where $(s_n)_{n \geq 1}$ is a sequence of \emph{shifts} in $\{-1, +1\}$.
This is somehow similar to a previous work of the author~\cite{PreprintSanna0}, in which least common multiples of the sequence of shifted Fibonacci numbers $(F_n + s_n)_{n \geq 1}$ were studied.

Our first result regards periodic sequences of shifts.

\begin{theorem}\label{thm:periodic}
Let $a \geq 2$ be an integer and let $\mathbf{s} = (s_n)_{n \geq 1}$ be a periodic sequence in $\{-1, +1\}$.
Then there exists an effectively computable rational number $C_\mathbf{s} > 0$ such that
\begin{equation}\label{equ:periodic}
\log\lcm(a + s_1, a^2 + s_2, \dots, a^n + s_n) \sim C_\mathbf{s} \cdot \frac{\log a}{\pi^2} \cdot n^2 ,
\end{equation}
as $n \to +\infty$.
\end{theorem}

We computed the constant $C_\mathbf{s}$ for periodic sequences with short period, see Table~\ref{tab:sCs}.

\begin{table}[h]
\label{tab:sCs}
\caption{Values of $C_\mathbf{s}$ for periodic sequences $\mathbf{s}$ of period at most $5$.}
\centering
\begin{tabular}{cc|cc|cc|cc}
  \toprule
  $\mathbf{s}$ & $C_{\mathbf{s}}$ & $\mathbf{s}$ & $C_{\mathbf{s}}$ & $\mathbf{s}$ & $C_{\mathbf{s}}$ & $\mathbf{s}$ & $C_{\mathbf{s}}$ \\\hline
  \texttt{-} & $3$ & \texttt{-+--} & $27/8$ & \texttt{--+-+} & $319/96$ & \texttt{+--+-} & $733/216$\rule{0pt}{2.6ex}\\
  \texttt{+} & $4$ & \texttt{-++-} & $125/36$ & \texttt{--++-} & $487/144$ & \texttt{+--++} & $769/216$\\
  \texttt{-+} & $4$ & \texttt{-+++} & $38/9$ & \texttt{--+++} & $7687/2160$ & \texttt{+-+--} & $487/144$\\
  \texttt{+-} & $3$ & \texttt{+---} & $3$ & \texttt{-+---} & $101/32$ & \texttt{+-+-+} & $7687/2160$\\
  \texttt{--+} & $13/4$ & \texttt{+--+} & $7/2$ & \texttt{-+--+} & $319/96$ & \texttt{+-++-} & $2123/576$\\
  \texttt{-+-} & $105/32$ & \texttt{+-++} & $7/2$ & \texttt{-+-+-} & $487/144$ & \texttt{+-+++} & $2219/576$\\
  \texttt{-++} & $173/48$ & \texttt{++--} & $125/36$ & \texttt{-+-++} & $7687/2160$ & \texttt{++---} & $487/144$\\
  \texttt{+--} & $105/32$ & \texttt{++-+} & $38/9$ & \texttt{-++--} & $733/216$ & \texttt{++--+} & $7687/2160$\\
  \texttt{+-+} & $173/48$ & \texttt{+++-} & $27/8$ & \texttt{-++-+} & $769/216$ & \texttt{++-+-} & $2123/576$\\
  \texttt{++-} & $47/12$ & \texttt{----+} & $19/6$ & \texttt{-+++-} & $2123/576$ & \texttt{++-++} & $2219/576$\\
  \texttt{---+} & $7/2$ & \texttt{---+-} & $101/32$ & \texttt{-++++} & $2219/576$ & \texttt{+++--} & $2123/576$\\
  \texttt{--+-} & $3$ & \texttt{---++} & $319/96$ & \texttt{+----} & $101/32$ & \texttt{+++-+} & $2219/576$\\
  \texttt{--++} & $7/2$ & \texttt{--+--} & $101/32$ & \texttt{+---+} & $319/96$ & \texttt{++++-} & $39/10$\\
  \bottomrule
\end{tabular}
\end{table}

Our second result is an almost sure asymptotic formula for random sequences of shifts (see~\cite{MR4091939, MR4220046, MR3239153, MR4009436} for similar results on least common multiples of random sequences). 

\begin{theorem}\label{thm:random}
Let $a \geq 2$ be an integer and let $(s_n)_{n \geq 1}$ be a sequence of independent and uniformly distributed random variables in $\{-1, +1\}$.
Then
\begin{equation}\label{equ:random}
\log\lcm(a + s_1, a^2 + s_2, \dots, a^n + s_n) \sim 6 \Li_2\!\big(\tfrac1{2}\big) \cdot \frac{\log a}{\pi^2} \cdot n^2 ,
\end{equation}
with probability $1 - o(1)$, as $n \to +\infty$, where $\Li_2(z) := \sum_{k = 1}^\infty z^k / k^2$ is the dilogarithm function. 
\end{theorem}

\begin{remark}
It is known that $\Li_2\!\big(\tfrac1{2}\big) = \big(\pi^2 - 6(\log 2)^2\big) / 12$ (see, e.g.,~\cite{MR2290758}), but in~\eqref{equ:random} we preferred to keep explicit the factor $6 \Li_2\!\big(\tfrac1{2}\big)$ in order to ease the comparison with~\eqref{equ:plus-formula}, \eqref{equ:minus-formula}, and~\eqref{equ:periodic}. 
\end{remark}

We leave the following questions to the interested reader:
\begin{question}\label{que:limit}
Is there a simple characterization of the set $\mathcal{E}$ of sequences $\mathbf{s} = (s_n)_{n \geq 1}$ in $\{-1, +1\}$ such that the limit
\begin{equation*}
L(\mathbf{s}) := \lim_{n \to +\infty} \frac{\log\lcm(a + s_1, a^2 + s_2, \dots, a^n + s_n)}{(\log a / \pi^2) \cdot n^2}
\end{equation*}
exists? (It follows from Lemma~\ref{lem:logellsn} below that $L(\mathbf{s})$ does not depend on $a$.)
\end{question}

\begin{question}
What is the image $L(\mathcal{E})$?
\end{question}

\begin{question}
Does (an appropriate normalization of) the random variable on the left-hand side of~\eqref{equ:random} converge to some known distribution?
\end{question}

\section{Notation}

We employ the Landau--Bachmann ``Big Oh'' and ``little oh'' notations $O$ and $o$, as well as the associated Vinogradov symbol $\ll$, with their usual meanings.
For real random variables $X$ and $Y$, depending on $n$, we say that ``$X \sim Y$ with probability $1 - o(1)$ as $n \to +\infty$'' if for every $\varepsilon > 0$ we have $\mathbb{P}\big[|X - Y| > \varepsilon|Y|\big] = o_\varepsilon(1)$ as $n \to +\infty$.
We let $[m, n]$ and $(m, n)$ denote the least common multiple and the greatest common divisor, respectively, of the two integers $m$ and $n$. 
We reserve the letter $p$ for prime numbers, and we let $\nu_p$ denote the $p$-adic valuation.
We write $\varphi(n)$ and $\tau(n)$ for the Euler function and the number of positive divisors, respectively, of a natural number $n$.

\section{Preliminaries}

Hereafter, let $a \geq 2$ be a fixed integer.
Define the \emph{$n$th cyclotomic polynomial} by
\begin{equation}\label{equ:cyclo}
\Phi_n(X) := \prod_{\substack{1 \,\leq\, k \,\leq\, n \\ (n, k) \,=\, 1}} \left(X - \mathrm{e}^{\frac{2 \pi \mathbf{i} k}{n}}\right) ,
\end{equation}
for every integer $n \geq 1$.
It is well known that $\Phi_n(X) \in \mathbb{Z}[X]$.
Moreover, from~\eqref{equ:cyclo} we get that
\begin{equation}\label{equ:shiftedan}
a^n - 1 = \prod_{d \,\in\, \mathcal{D}^-(n)} \Phi_d(a) \quad \text{ and } \quad a^n + 1 = \prod_{d \,\in\, \mathcal{D}^+(n)} \Phi_d(a) ,
\end{equation}
for every integer $n \geq 1$, where $\mathcal{D}^-(n) := \{d \in \mathbb{N} : d \mid n\}$ and $\mathcal{D}^+(n) := \mathcal{D}(2n) \setminus \mathcal{D}(n)$.

We need two results about the sequence of integers $(\Phi_d(a))_{d \in \mathbb{N}}$.

\begin{lemma}\label{lem:gcdPhi}
We have $(\Phi_m(a), \Phi_n(a)) \mid m$, for all integers $m > n \geq 1$.
\end{lemma}
\begin{proof}
If $n > 1$ then the claim follows from~\cite[Theorem~3.1]{MR1416242}.
If $n = 1$ then notice that $\Phi_n(a) = a - 1$ and so $a \equiv 1 \pmod d$, where $d := (\Phi_m(a), \Phi_n(a))$.
Therefore,
\begin{equation*}
d \mid \Phi_m(a) \mid \frac{a^m - 1}{a - 1} \equiv 1 + a + a^2 + \cdots + a^{m-1} \equiv m \pmod d , 
\end{equation*}
and consequently $d$ divides $m$.
\end{proof}

\begin{lemma}\label{lem:asympPhi}
We have $\log \Phi_n(a) = \varphi(n) \log a + O_a(1)$, for every integer $n \geq 1$.
\end{lemma}
\begin{proof}
See~\cite[Lemma~2.1(iii)]{MR4003803}.
\end{proof}

For every sequence $\mathbf{s} = (s_n)_{n \geq 1}$ in $\{-1, +1\}$, let us define
\begin{equation*}
\ell_{a,\mathbf{s}}(n) := \lcm(a + s_1, a^2 + s_2, \dots, a^n + s_n) 
\end{equation*}
and
\begin{equation*}
\mathcal{L}_\mathbf{s}(n) := \bigcup_{k \,\leq\, n} \mathcal{D}^{(s_k)}(k) ,
\end{equation*}
for all integers $n \geq 1$.

The next lemma will be fundamental in the proofs of Theorem~\ref{thm:periodic} and Theorem~\ref{thm:random}.

\begin{lemma}\label{lem:logellsn}
We have
\begin{equation*}
\log \ell_{a,\mathbf{s}}(n) = \sum_{d \,\in\, \mathcal{L}_\mathbf{s}(n)} \varphi(d) \log a + O_a\!\left(\frac{n^2}{\log n}\right) ,
\end{equation*}
for every integer $n \geq 2$.
\end{lemma}
\begin{proof}
Suppose that $p^v \mid\mid \ell_{a,\mathbf{s}}(n)$, for some prime number $p \leq 2n$ and some integer $v \geq 1$.
Then $p^v \mid a^k + s_k$ for some positive integer $k \leq n$, and consequently $p^v \leq a^{n + 1}$.
Therefore,
\begin{equation}\label{equ:smallprimes}
\log\!\Big(\prod_{\substack{p^v \,\mid\mid\, \ell_{a,\mathbf{s}}(n) \\ p \,\leq\, 2n}} p^v \Big) \leq \log\!\Big(\prod_{\substack{p^v \,\mid\mid\, \ell_{a,\mathbf{s}}(n) \\ p \,\leq\, 2n}} a^{n + 1} \Big) \leq \#\{p : p \leq 2n\} \cdot (n + 1) \log a \ll_a \frac{n^2}{\log n} , 
\end{equation}
since the number of primes not exceeding $2n$ is $O(n / \!\log n)$.

In light of Lemma~\ref{lem:gcdPhi}, the integers $\Phi_1(a), \dots, \Phi_{2n}(a)$ are pairwise coprime except for prime factors not exceeding $2n$.
Hence, the identities~\eqref{equ:shiftedan} together with~\eqref{equ:smallprimes} yield
\begin{align*}
\log \ell_{a,\mathbf{s}}(n) &= \log\!\Big(\prod_{d \,\in\, \mathcal{L}_\mathbf{s}(n)} \Phi_d(a) \Big) + O_a\!\left(\frac{n^2}{\log n}\right) \\
 &= \sum_{d \,\in\, \mathcal{L}_\mathbf{s}(n)} \varphi(d)\log a + O_a(\#\mathcal{L}_\mathbf{s}(n)) + O_a\!\left(\frac{n^2}{\log n}\right) ,
\end{align*}
where we used Lemma~\ref{lem:asympPhi}.
The claim follows since $\mathcal{L}_\mathbf{s}(n) \subseteq [1, 2n]$ and so $\#\mathcal{L}_\mathbf{s}(n) \leq 2n$.
\end{proof}

For all integers $r \geq 0$ and $m \geq 1$, and for every $x \geq 1$, let us define the arithmetic progression
\begin{equation*}
\mathcal{A}_{r, m}(x) := \big\{n \leq x : n \equiv r \!\!\!\!\pmod m\big\} .
\end{equation*}
We need an asymptotic formula for a sum of the Euler totient function over an arithmetic progression.

\begin{lemma}\label{lem:phisum}
Let $r, m$ be positive integers and let $z \in {[0, 1)}$.
Then we have
\begin{equation*}
\sum_{n \,\in\, \mathcal{A}_{r,m}(x)} \varphi(n)\big(1 - z^{\lfloor x / n \rfloor}\big) = \frac{3}{\pi^2} \cdot c_{r,m} \cdot \frac{(1 - z) \Li_2(z)}{z} \cdot x^2 + O_{r,m}\big(x (\log x)^2\big) ,
\end{equation*}
for every $x \geq 2$, where
\begin{equation*}
c_{r, m} := \frac1{m} \prod_{\substack{p \,\mid\, m \\ \!\!p \,\mid\, r}} \left(1 + \frac1{p}\right)^{-1} \prod_{\substack{p \,\mid\, m \\ \!\! p \,\nmid\, r}} \left(1 - \frac1{p^2}\right)^{-1} ,
\end{equation*}
while for $z = 0$ the factor involving $\Li_2(z)$ is meant to be equal to $1$, and the error term can be improved to $O_{r,m}(x \log x)$. 
\end{lemma}
\begin{proof}
See~\cite[Lemma~3.4, Lemma~3.5]{PreprintSanna0}.
\end{proof}

\section{Proof of Theorem~\ref{thm:periodic}}

We need a lemma about unions of $\mathcal{D}^-(k)$, respectively $\mathcal{D}^+(k)$, with $k \in \mathcal{A}_{r, m}(x)$.

\begin{lemma}\label{lem:Dunions}
Let $r, m \geq 1$ be integers and let $u \in \{-1, +1\}$.
Then there exist a set $\mathcal{T}_{r,m}^{(u)} \subseteq \{1, \dots, 2m\}$ and a rational number $\theta_{r,m}^{(u)} > 0$, both effectively computable, such that
\begin{equation*}
\bigcup_{k \,\in\, \mathcal{A}_{r,m}(x)} \mathcal{D}^{(u)}(k) = \bigcup_{t \,\in\, \mathcal{T}_{r,m}^{(u)}} \mathcal{A}_{t, 2m}\big(\theta_{r,m}^{(u)} x\big) ,
\end{equation*}
for every $x \geq 1$.
\end{lemma}
\begin{proof}
The claim follows from~\cite[Lemma~3.2, Lemma~3.3]{PreprintSanna0} and taking into account that $\mathcal{A}_{t, m}(x) = \mathcal{A}_{t, 2m}(x) \cup \mathcal{A}_{t + m, 2m}(x)$.
\end{proof}

\begin{proof}[Proof of Theorem~\ref{thm:periodic}]
Let $\mathbf{s} = (s_n)_{n \geq 1}$ be a periodic sequence in $\{-1, +1\}$, and let $m$ be the length of its period.
Moreover, let $\mathcal{R}_\mathbf{s}^{(u)} := \{r \in \{1, \dots, m\} : s_r = u\}$ for $u \in \{-1, +1\}$.
By periodicity of $\mathbf{s}$ and by Lemma~\ref{lem:Dunions}, it follows that
\begin{align*}
\mathcal{L}_\mathbf{s}(n) &= \bigcup_{u \,\in\, \{-1, +1\}} \bigcup_{r \,\in\, \mathcal{R}_\mathbf{s}^{(u)}} \bigcup_{k \,\in\, \mathcal{A}_{r,m}(n)} \mathcal{D}^{(u)}(k) \\
 &= \bigcup_{u \,\in\, \{-1, +1\}} \bigcup_{r \,\in\, \mathcal{R}_\mathbf{s}^{(u)}} \bigcup_{t \,\in\, \mathcal{T}_{r, m}^{(u)}} \mathcal{A}_{t, 2m}\big(\theta_{r,m}^{(u)}n\big) \\
 &= \bigcup_{t \,\in\, \mathcal{T}_\mathbf{s}} \mathcal{A}_{t, 2m}(\theta_{\mathbf{s},t} n)
\end{align*}
where
\begin{equation*}
\mathcal{T}_\mathbf{s} := \bigcup_{u \,\in\, \{-1, +1\}} \bigcup_{r \,\in\, \mathcal{R}_\mathbf{s}^{(u)}} \mathcal{T}_{r, m}^{(u)}
\end{equation*}
and
\begin{equation*}
\theta_{\mathbf{s},t} := \max\!\left\{\theta_{r,m}^{(u)} : t \in \mathcal{T}_{r, m}^{(u)} \text{ for some } u \in \{-1, +1\}, \, r \in \mathcal{R}_\mathbf{s}^{(u)} \right\} .
\end{equation*}
Hence, from Lemma~\ref{lem:logellsn} and Lemma~\ref{lem:phisum} (with $z = 0$), we get that
\begin{align*}
\log \ell_{a,\mathbf{s}}(n) &= \sum_{t \,\in\, \mathcal{T}_\mathbf{s}} \; \sum_{d \,\in\, \mathcal{A}_{t, 2m}(\theta_{\mathbf{s}, t} n)} \varphi(d) \log a + O_a\!\left(\frac{n^2}{\log n}\right) \\
&= C_\mathbf{s} \cdot \frac{\log a}{\pi^2} \cdot n^2 + O_{a, m}\!\left(\frac{n^2}{\log n}\right) ,
\end{align*}
where
\begin{equation*}
C_\mathbf{s} := 3 \sum_{t \,\in\, \mathcal{T}_\mathbf{s}} c_{t, 2m} \theta_{\mathbf{s}, t}^2 
\end{equation*}
is a positive rational number effectively computable in terms of $s_1, \dots, s_m$.
\end{proof}

\section{Proof of Theorem~\ref{thm:random}}

Let $\mathbf{s} = (s_n)_{n \geq 1}$ be a sequence of independent and uniformly distributed random variables in $\{-1, +1\}$.
Moreover, define
\begin{equation*}
I_\mathbf{s}(n, d) := \begin{cases} 
 1 & \text{ if } d \in \mathcal{L}_\mathbf{s}(n); \\
 0 & \text{ otherwise};
\end{cases}
\end{equation*}
for all integers $n, d \geq 1$.
The next lemma gives two expected values involving $I_\mathbf{s}(n, d)$.

\begin{lemma}\label{lem:EIsd}
We have
\begin{equation}\label{equ:EIsd}
\mathbb{E}\big[I_\mathbf{s}(n, d)\big] = 1 - 2^{-\lfloor n (2, d) / d \rfloor}
\end{equation}
and
\begin{align*}
\mathbb{E}&\big[I_\mathbf{s}(n, d_1)I_\mathbf{s}(n, d_2)\big] = 1 - 2^{-\lfloor n(2,d_1)/d_1 \rfloor} - 2^{-\lfloor n(2,d_2)/d_2 \rfloor} \\
 &+ 2^{-\lfloor n (2, d_1) / d_1 \rfloor - \lfloor n (2, d_2) / d_2 \rfloor + \lfloor n (2, [d_1, d_2]) / [d_1, d_2] \rfloor} \begin{cases} 1 & \text{ if } \nu_2(d_1) = \nu_2(d_2); \\ 0 & \text{ otherwise}; \end{cases}
\end{align*}
for all integers $d, d_1, d_2 \geq 1$.
\end{lemma}
\begin{proof}
On the one hand, by the definitions of $I_\mathbf{s}(n, d)$ and $\mathcal{L}_\mathbf{s}(n)$, we have
\begin{align*}
\mathbb{E}\big[I_\mathbf{s}(n, d)\big] &= \mathbb{P}\big[d \in \mathcal{L}_\mathbf{s}(n)\big] = 1 - \mathbb{P}\!\left[\bigwedge_{\substack{k \,\leq\, n \\[1pt] \;d \,\mid\, 2k}} \big((d \mid k \land s_k = +1) \lor (d \nmid k \land s_k = -1)\big)\right] \\
 &= 1 - 2^{-\#\{k \,\leq\, n \,:\, d \,\mid\, 2k\}} = 1 - 2^{-\lfloor n (2, d) / d \rfloor} ,
\end{align*}
which is~\eqref{equ:EIsd}.

On the other hand, by linearity of the expectation and by~\eqref{equ:EIsd}, we have
\begin{align}\label{equ:EIsd1Isd2P}
\mathbb{E}\big[I_\mathbf{s}(n, d_1)I_\mathbf{s}(n, d_2)\big] &= \mathbb{E}\big[I_\mathbf{s}(n, d_1) + I_\mathbf{s}(n, d_2) - 1 + \big(1 - I_\mathbf{s}(n, d_1)\big)\big(1 - I_\mathbf{s}(n, d_2)\big)\big] \\
 &= \mathbb{E}\big[I_\mathbf{s}(n, d_1)\big] + \mathbb{E}\big[I_\mathbf{s}(n, d_2)\big] - 1 + \mathbb{E}\big[\big(1 - I_\mathbf{s}(n, d_1)\big)\big(1 - I_\mathbf{s}(n, d_2)\big)\big] \nonumber \\
 &= 1 - 2^{-\lfloor n(2,d_1)/d_1 \rfloor} - 2^{-\lfloor n(2,d_2)/d_2 \rfloor} + \mathbb{P}\big[d_1 \notin \mathcal{L}_\mathbf{s}(n) \land d_2 \notin \mathcal{L}_\mathbf{s}(n)\big] . \nonumber
\end{align}
Let $P$ be the probability at the end of~\eqref{equ:EIsd1Isd2P}.

Suppose for a moment that $[d_1, d_2] \leq 2n$ and that $d_1$ and $d_2$ have different $2$-adic valuations, say $\nu_2(d_1) < \nu_2(d_2)$, without loss of generality.
Let $h := [d_1, d_2] / 2$ and note that $h$ is an integer not exceeding $n$.
Furthermore, $d_1 \in \mathcal{D}^-(h)$ and $d_2 \in \mathcal{D}^+(h)$.
Hence, no matter the value of $s_h$, at least one of $d_1, d_2$ belongs to $\mathcal{L}_\mathbf{s}(n)$, and consequently $P = 0$.

Now suppose that $[d_1, d_2] > 2n$ or $\nu_2(d_1) = \nu_2(d_2)$.
In the second case, note that for every integer $k$ such that $[d_1, d_2] \mid 2k$ we have that either $d_1 \mid k$ and $d_2 \mid k$, or $d_1 \nmid k$ and $d_2 \nmid k$.
Therefore,
\begin{align*}
P &= \mathbb{P}\!\left[\bigwedge_{\substack{k \,\leq\, n \\ d_1 \,\mid\, 2k \,\land\, d_2 \,\nmid\, 2k}} \big((d_1 \mid k \land s_k = +1) \lor (d_1 \nmid k \land s_k = -1)\big) \right. \\
 &\phantom{mmm}\land \bigwedge_{\substack{k \,\leq\, n \\ d_1 \,\nmid\, 2k \,\land\, d_2 \,\mid\, 2k}} \big((d_2 \mid k \land s_k = +1) \lor (d_2 \nmid k \land s_k = -1)\big) \\
 &\phantom{mmm}\left.\land \bigwedge_{\substack{k \,\leq\, n \\ d_1 \,\mid\, 2k \,\land\, d_2 \,\mid\, 2k}} \big((d_1 \mid k \land d_2 \mid k \land s_k = +1) \lor (d_1 \nmid k \land d_2 \nmid k \land s_k = -1)\big) \right] \\
 &= 2^{-\#\{k \,\leq\, n \,:\, d_1 \,\mid\, 2k \,\lor\, d_2 \,\mid\, 2k\}} \\
 &= 2^{-\lfloor n (2, d_1) / d_1 \rfloor - \lfloor n (2, d_2) / d_2 \rfloor + \lfloor n (2, [d_1, d_2]) / [d_1, d_2] \rfloor} ,
\end{align*} 
and the proof is complete.
\end{proof}

The following lemma is a simple upper bound for a sum of greatest common divisors.

\begin{lemma}\label{lem:gcdsum}
We have
\begin{equation*}
\sum_{[d_1\!,\, d_2] \,\leq\, n} (d_1, d_2) \ll n^2 ,
\end{equation*}
for every integer $n \geq 1$.
\end{lemma}
\begin{proof}
Let $a_i := d_i / d$ for $i=1,2$, where $d := (d_1, d_2)$.
Then we have
\begin{align*}
\sum_{[d_1\!,\, d_2] \,\leq\, n} (d_1, d_2) &= \sum_{d \,\leq\, n} d \sum_{\substack{a_1 a_2 \,\leq\, n / d \\ (a_1\!,\, a_2) \,=\, 1}} 1 \leq \sum_{d \,\leq\, n} d \sum_{m \,\leq\, n / d} \tau(m) \\
 &\ll n \sum_{d \,\leq\, n} \log\!\left(\frac{n}{d}\right) = n \left(n \log n - \log(n!)\right) < n^2,
\end{align*}
where we used the upper bound $\sum_{m \leq x} \tau(m) \ll x \log x$ (see, e.g.,~\cite[Ch.~I.3, Theorem~3.2]{MR3363366}) and the inequality $n! > (n / \mathrm{e})^n$.
\end{proof}

\begin{proof}[Proof of Theorem~\ref{thm:random}]
Let us define the random variable
\begin{equation*}
X := \sum_{d \,\leq\, 2n} \varphi(d) \, I_\mathbf{s}(n, d) .
\end{equation*}
From the linearity of expectation, Lemma~\ref{lem:EIsd}, and Lemma~\ref{lem:phisum}, it follows that
\begin{align}\label{equ:EX}
\mathbb{E}[X] &:= \sum_{d \,\leq\, 2n} \varphi(d) \, \mathbb{E}\big[I_\mathbf{s}(n, d)] \\
 &= \sum_{d \,\leq\, 2n} \varphi(d) \left(1 - 2^{-\lfloor n (2, d) / d \rfloor} \right) \nonumber \\
 &= \sum_{d \,\in\, \mathcal{A}_{1,2}(n)} \varphi(d) \left(1 - 2^{-\lfloor n / d \rfloor} \right) + \sum_{d \,\in\, \mathcal{A}_{2,2}(2n)} \varphi(d) \left(1 - 2^{-\lfloor 2n / d \rfloor} \right) \nonumber \\
 &= \frac{3}{\pi^2} \left(c_{1,2} + 4 c_{2,2}\right) \Li_2\!\big(\tfrac1{2}\big) \, n^2 + O\big(n(\log n)^2) \nonumber \\
 &= \frac{6}{\pi^2} \Li_2\!\big(\tfrac1{2}\big) \, n^2  + O\big(n(\log n)^2\big) . \nonumber
\end{align}
Furthermore, from Lemma~\ref{lem:EIsd} and Lemma~\ref{lem:gcdsum}, we get that
\begin{align}\label{equ:VX}
& \mathbb{V}[X] = \mathbb{E}\big[X^2\big] - \mathbb{E}[X]^2 \\
 &= \sum_{d_1\!,\, d_2 \,\leq\, 2n} \varphi(d_1) \,\varphi(d_2) \Big(\mathbb{E}\big[I_\mathbf{s}(n, d_1)I_\mathbf{s}(n, d_2)\big] - \mathbb{E}\big[I_\mathbf{s}(n, d_1)\big] \mathbb{E}\big[I_\mathbf{s}(n, d_2)\big]\Big) \nonumber \\
 &\leq \sum_{[d_1\!,\, d_2] \,\leq\, 2n} d_1 d_2 \, 2^{-\lfloor n (2, d_1) / d_1 \rfloor - \lfloor n (2, d_2) / d_2 \rfloor + \lfloor n (2, [d_1, d_2]) / [d_1, d_2] \rfloor} \left(1 - 2^{-\lfloor n (2, [d_1, d_2]) / [d_1, d_2] \rfloor}\right) \nonumber \\
 &\leq \sum_{[d_1\!,\, d_2] \,\leq\, 2n} d_1 d_2 \left\lfloor \frac{n(2, [d_1, d_2])}{[d_1, d_2]} \right\rfloor \ll n \sum_{[d_1\!,\, d_2] \,\leq\, 2n} (d_1, d_2) \ll n^3 , \nonumber
\end{align}
where we also used the inequality $1 - 2^{-k} \leq k / 2$, which holds for every integer $k \geq 0$.

Therefore, by Chebyshev's inequality, \eqref{equ:EX}, and~\eqref{equ:VX}, it follows that
\begin{equation*}
\mathbb{P}\Big[\big|X - \mathbb{E}[X]\big| > \varepsilon \, \mathbb{E}[X] \Big] \leq \frac{\mathbb{V}[X]}{\big(\varepsilon \mathbb{E}[X]\big)^2} \ll \frac1{\varepsilon^2 n} = o_\varepsilon(1) ,
\end{equation*}
as $n \to +\infty$.
Hence, again by~\eqref{equ:EX}, we have
\begin{equation*}
X \sim \mathbb{E}[X] \sim \frac{6}{\pi^2} \Li_2\!\big(\tfrac1{2}\big) n^2 ,
\end{equation*} 
with probability $1 - o(1)$.
Finally, thanks to Lemma~\ref{lem:logellsn}, we have
\begin{equation*}
\log \ell_{a,\mathbf{s}}(n) = X \log a + O_a\!\left(\frac{n^2}{\log n}\right) ,
\end{equation*}
and the asymptotic formula~\eqref{equ:random} follows.
\end{proof}

\bibliographystyle{amsplain}

\end{document}